\documentclass[12pt,reqno]{amsart}
\usepackage{amsmath}
\usepackage{amssymb, latexsym, amsfonts, amscd, amsthm, mathrsfs, enumerate, esint}
\usepackage[usenames,dvipsnames]{color}
\usepackage[all]{xy}
\usepackage{graphicx}
\usepackage{epsfig}
\usepackage{hyperref}

\numberwithin{equation}{section}

\definecolor{purple}{rgb}{0.9,0,0.8}

\definecolor{gray}{rgb}{0.7,0.7,0.7}

\newcommand{\abbr}[1]{{\sc\lowercase{#1}}}

\newtheorem{thm}{Theorem}[section]
\newtheorem{cor}[thm]{Corollary}
\newtheorem{lem}[thm]{Lemma}
\newtheorem{ppn}[thm]{Proposition}

\theoremstyle{definition}
\newtheorem{defn}[thm]{Definition}

\newtheorem{remark}[thm]{Remark}

\newcommand{\beq}{\begin{equation}}
\newcommand{\eeq}{\end{equation}}

\newcommand{\ep}{\epsilon}

\newcommand{\B}{\mathbb{B}}

\newcommand{\C}{\mathbb{C}}

\newcommand{\D}{\mathbb{D}}

\newcommand{\E}{\mathbb{E}}

\newcommand{\F}{\mathbb{F}}

\newcommand{\G}{\mathbb{G}}

\newcommand{\K}{\mathbb{K}}

\newcommand{\N}{\mathbb{N}}

	\renewcommand{\P}{\mathbb{P}}

\newcommand{\R}{\mathbb{R}}

\newcommand{\Z}{\mathbb{Z}}

\newcommand{\ovl}{\overline}

\newcommand{\cE}{\mathcal{E}}
\newcommand{\cF}{\mathcal{F}}

\newcommand{\cS}{\mathcal{S}}

\newcommand{\cV}{\mathcal{V}}

\newcommand{\wt}{\widetilde}

\renewcommand{\emptyset}{\varnothing}

\begin{document}
\title{Growing in time IDLA cluster is recurrent}
\author{Ruojun Huang}
\address{Courant Institute of Mathematical Sciences, USA.}

\keywords{Isoperimetry, internal diffusion limited aggregation, recurrence, random walk, random environment, changing environment}
\subjclass[2010]{Primary 60K37, 82C24; Secondary 60K35.}
\date{\today}
\maketitle
\begin{abstract}
We show that Internal Diffusion Limited Aggregation (\abbr{IDLA}) on $\Z^d$ has near optimal Cheeger constant when the growing cluster is large enough. This implies,  through a heat kernel lower bound derived previously in \cite{H}, that simple random walk evolving independently on growing in time \abbr{IDLA} cluster is recurrent when $d\ge 3$.
\end{abstract}

\section{Introduction}
This work is concerned with proving recurrence of simple random walk on certain type of growing clusters. Specifically, we consider the Internal Diffusion Limited Aggregation (\abbr{IDLA} \cite{LBG}) $\{\D_n\}_{n\in\N}$ on $\Z^d$, $d\ge 3$ in discrete time, that is, starting from $\D_0=\{\mathbf 0\}$ consisting of the origin, one iteratively grows the cluster by adding to $\D_n$ the first vertex $v_n\in \Z^d\backslash\D_n$ that is hit by an independent particle injected from the origin performing discrete-time simple random walk (\abbr{SRW}), and set $\D_{n+1}=\D_n\cup\{v_n\}$, for every $n\in\N$. As such, $\D_n$ are random,  finite connected vertex subsets of $\Z^d$, monotonically increasing with $|\D_n|=n$, where we denote by $|\cdot|$ the cardinality of a set. Being a canonical model of Laplacian growth, \abbr{IDLA} has been much studied, cf. \cite{DF, LBG, GQ, AG1, AG2, JLS1, JLS2} and references therein, with the well-known shape and fluctuation theorems stating, in particular, that for each $d\ge 3$, there exists finite constant $a=a(d)$, such that almost surely for all $n=\lfloor\omega_dr^d\rfloor$ large enough, we have that $\B_{r-a\sqrt{\log r}}\subset \D_n\subset \B_{r+a\sqrt{\log r}}$, where $\omega_d$ denotes the volume of a unit ball in $\R^d$ and $\B_R=\mathbf B_R\cap\Z^d$ for the origin-centered Euclidean ball $\mathbf{B}_R$ of radius $R$. In this article, we view each $\D_n$ as a subgraph of $\Z^d$ imposing an edge  between any two vertices of $\Z^d$-distance $1$ within the cluster, and the inclusion above holds in the sense of graphs. 

Consider now an (autonomous) simple random walk  $\{X_n\}_{n\in\N}$ evolving independently on $\{\D_n\}_{n\in \N}$ starting from $X_0=\mathbf 0$. That is, fixing a realization $\omega$ of the \abbr{IDLA} cluster $\{\D^\omega_n\}_{n\in\N}$, at each step $n$, $X_n\in\D^\omega_n$ chooses uniformly among all vertices in $\D^\omega_n$ having $\Z^d$-distance $1$ from it and passes to $X_{n+1}$, before domain changes from $\D^\omega_n$ to $\D^\omega_{n+1}$. There is at least one vertex available for $X_n$ and at most $2d$. Fixing $\omega$, such random walk $\{X_n\}_{n\in\N}$ is a time-inhomogeneous Markov chain, albeit at perpetual non-equilibrium, but we are interested in whether the walk is recurrent to the origin, and whether recurrence holds for almost every realization $\omega$ of its environment. Indeed, we have studied this problem within a more general set-up in \cite{DHS}, in particular showing that if one considers a continuous-time version of \abbr{IDLA} $\{\D_t\}_{t\ge 0}$,  where particles are injected from the origin at an inhomogeneous Poisson rate, then on $\Z^d$, $d\ge 3$, for a.e. $\omega$, the (discrete-time) \abbr{SRW} $\{X_{\lfloor t \rfloor}\}_{t\ge 0}$ returns to its starting point finitely many times with probability one as soon as  $\int_1^\infty|\D^\omega_t|^{-1}dt$ is finite. However, it is only conjectured that for constant injection rate of particles, namely when $|\D_t|$ grows approximately linearly, w.p.1 $\{X_{\lfloor t \rfloor}\}_{t\ge 0}$ visits its starting point infinitely often as $\int_1^\infty |\D^\omega_t|^{-1}dt$ then diverges. See Section \ref{sec:cont} where we revisit this problem, for details. Similarly, the latter outcome is expected for the canonical discrete-time \abbr{IDLA} dynamics in which $|\D_n|=n$. The reason for such recurrent cases to be open is that one needs to rule out the possibility that the {\it evolving} \abbr{IDLA} boundary be too rough and traps the random walk, thus delaying its return to the bulk. This is not automatically ruled out by the limiting shape. Besides, in the time-inhomogeneous setting one loses some robust tools such as commute time identity from electrical networks. 

It turns out that the issue can be resolved by understanding the isoperimetry property of a typical large cluster.
In a previous work \cite{H}, we give sufficient conditions for recurrence (and transience) of \abbr{SRW} evolving independently  on sequences of growing subgraphs of a pre-given infinite graph, in terms of the behaviors of the Cheeger (i.e. isoperimetric) constant of each subgraph in that sequence. When certain sufficient condition is satisfied,  see \eqref{regul}, the heat kernel of $X_n$ at its starting point can be bounded below by $c/|\D_n|$ at time $n$, for some absolute constant $c>0$. This yields corresponding lower bounds on the expected occupation time (i.e. Green's function), upon integrating the heat kernel. The framework includes as special case the \abbr{IDLA} cluster on $\Z^d$, as soon as the cluster is large enough.  In this work, we show that the isoperimetry of large \abbr{IDLA} cluster is near optimal, namely it differs from the isoperimetry of a lattice ball by at most a logarithmic correction, see Theorem \ref{thm:cheeger}, which implies recurrence of the random walk, see Corollary \ref{rec}. To derive the former, we rely heavily on the known limit shape and fluctuation bounds from circularity. The relation between isoperimetry and heat kernel or mixing time estimates on fixed graphs is well understood, cf. \cite{B, PS, SC}. Even in random environments such as supercritical percolation clusters, a key step in deriving the latter estimates is to understand the precise isoperimetric inequality satisfied by the cluster at large scales, see \cite{BM, MR, CF, PRS}. Here we see something similar, though in a dynamic setting.

\section{Framework and main results}    \label{sec:frame}
We work first in discrete time, denoting by $\{\D_n\}_{n\in \N}$ the growing \abbr{IDLA} cluster on $\Z^d$. We recall its growth mechanism: starting at $\D_0=\{\mathbf 0\}$, the {\it random} cluster grows by adding one vertex per integer time, chosen by an independent particle, denoted $Y_n$, doing simple random walk from the origin till visiting the first vertex $v_n$ outside of cluster, then we set $\D_{n+1}=\D_n\cup\{v_n\}$. Hence, $|\D_n|=n$. While respecting its growth mechanism, we omit particle $Y_n$'s travel time and consider it settled on $\Z^d\backslash \D_n$ instantaneously upon injected, for every $n\in\N$. Each $\D_n$ is then viewed as a subgraph of the lattice graph $\Z^d$, by imposing an edge between any $x, y\in \D_n$ of $\Z^d$-distance $1$, which we denote $x\sim y$. Further, we denote by $\B(v,R)=\mathbf{B}(v, R)\cap\Z^d$, the graph obtained from the intersection of $\Z^d$ lattice graph and an {\it open} Euclidean ball of radius $R>0$ centered at $v\in\Z^d$, with $\B_R:=\B(\mathbf 0,R)$ in case the center is the origin. 

Now, we recall the shape and fluctuation theorems for $\{\D_n\}$. The formulation below is taken from \cite[Theorem 0.1]{JLS2}, \cite[Theorem 1]{JLS1}.
\begin{thm}[\cite{LBG, AG1, AG2, JLS1, JLS2}] \label{shape-fluct}
For $d\ge 3$, there exist some finite constant $a=a(d)$ and integer $n_0$ such that for all $n=\lfloor\omega_d r^d\rfloor\ge n_0$, we have that
\begin{align}   \label{shape-event}
\P(\cS_n)\ge 1-n^{-2},   \qquad \text{where } \quad  \cS_n:=\{\B_{r-a\sqrt{\log r}}\subset \D_n\subset \B_{r+a\sqrt{\log r}}\}.
\end{align}
For $d=2$, the same holds with $\log r$ replacing $\sqrt{\log r}$ in $\cS_n$.
\end{thm}

We also require some notation for general graphs. To this end, given a graph $G$, we write $\cV(G)$ its vertex set, and $\cE(G)$ its edge set. Throughout, we write $|G|$ for $|\cV(G)|$, for short.
We now specify a notion of vertex boundary that we extensively use. 
\begin{defn}  
Given a graph $G$, the {\it inner vertex boundary} of $A\subset\cV(G)$ {\it relative to $G$} is defined to be
\begin{align}   \label{vert-bdry}
\partial^GA:=\{x\in A:\, \exists y\in \cV(G)\backslash A \text{ such that } (x,y)\in\cE(G)\}.
\end{align}
We omit the superscript $G$ in $\partial^GA$ in case $G=\Z^d$.
\end{defn}
Define the Cheeger constant (aka bottleneck ratio) of a finite graph $G$ as 
\begin{align}  \label{bottleneck}
\Phi(G)=\inf_{A\subset \cV(G), \, |A|\le |G|/2}\Big\{\frac{|\partial^GA|} {|A|}\Big\}.
\end{align}
This quantity is a well-known gauge for the connectivity of a finite set, in particular governing the mixing rates for reversible Markov chains on the set, see \cite{LPW}. Instead of the vertex boundary  in \eqref{bottleneck}, more commonly one uses the edge boundary (that is, the set of all edges with one end in $A$, the other in $\cV(G)\backslash A$); in our setting, they differ by at most a factor of $(2d-1)$, which is unimportant for us. In the sequel, we will be considering \eqref{bottleneck} for $G=\D_n$.

In the case of lattice balls $\B_R\subset\Z^d$, it is well known that there exists some positive constant $c_d$, such that $\Phi(\B_R)\ge c_d|\B_R|^{-1/d}$ for all $R>1$, and such isoperimetry is optimal for finite subgraphs of $\Z^d$. Given Theorem \ref{shape-fluct}, it is pertinent to wonder if the isoperimetry of lattice balls is inherited by \abbr{IDLA}, and the main result of this article is the following. 

\begin{thm} \label{thm:cheeger}
Fix $d\ge 3$. There exists some positive constant $c=c(d)$, such that $\P$-a.s. for all $n$ large enough, we have that 
\begin{align}   \label{lower-ch}
\Phi(\D_n)\ge cn^{-1/d}(\log n)^{-3/2}. 
\end{align}
\end{thm}
\begin{remark}
The result does not follow automatically from Theorem \ref{shape-fluct}, since with high probability $|\D_n\backslash\B_{r-a\sqrt{\log r}}|=O(n^{(d-1)/d}\sqrt{\log n})$, and {\it a-priori} a cut in $\D_n$ can be as small as consisting of only one edge. The point here is to show that randomly produced cluster, here \abbr{IDLA}, is much better structured than deterministic ones under the same shape constraints.

Nevertheless, we believe that the logarithmic correction in \eqref{lower-ch} should not be there. For that reason, we do not state the result for $d=2$.
\end{remark}

Turning to the application to recurrence on $\{\D_n\}$ we have in mind, denote by $\{X_n\}_{n\in\N}$ a {\it lazy} \abbr{SRW} on $\{\D_n\}$ with $X_0=\mathbf 0$. Fixing a realization $\omega$ of $\{\D_n\}$, its $n$-th step transition probability is defined to be 
\begin{align}           \label{srw-def}
\mathbf{P}^\omega(X_{n+1}=y|X_n=x)=
\begin{cases}
(1-\gamma^\omega_n(x))\cdot|\{z\in \D^\omega_n:\, x\sim z\}|^{-1}, \quad &\text{if }x\sim y\in \D^\omega_n, \, x\neq y \\
\gamma_n^\omega(x), \quad & \text{if }x=y\in\D^\omega_n        
\end{cases}
\end{align}
for some $\{\gamma_n^\omega(x)\}$ denoting the probability of staying put.
In words, each step the random walk either stays put or chooses uniformly among neighbors in its current cluster.
Here, $\mathbf{P}^\omega$ denotes the quenched probability governing the random walk $X$, and $\P$ is reserved for the randomness of $\{\D_n\}$. We will add the superscript $\omega$ only when it is helpful to emphasize the randomness of the environment.
We say that $\{X_n\}$ is uniformly lazy with probability $\gamma\in(0,1)$, if $\gamma^\omega_n(x)\ge \gamma$ for all $\omega$, $n$ and $x\in \D^\omega_n$. 
A special case of \cite[Proposition 1.13]{H} in our context is the following heat kernel lower bound for $\{X_n\}$ on $\{\D_n\}$, where the uniform laziness is but a technical assumption.

\begin{ppn}[\cite{H}]   \label{hk-lbd}
Suppose $\{X_n\}$ is uniformly lazy with probability $\gamma\in(0,1)$, and the following conditions on $\{\D_n\}$ are $\P$-a.s. satisfied: \\
(i). There exist some deterministic sequence $\{r(n)\}$, and a.s. finite random variable $\ovl{n} (\omega)$ such that 
\begin{align*}
\D_n&\supset \B_{r(n)}, \quad \forall n \ge \ovl n(\omega)\\
\text{and }\quad \lim_{n\to \infty}&\left\{\frac{|\D_n-\B_{r(n)}|}{|\D_n|}\right\}=0.
\end{align*}
(ii). 
\begin{align} \label{regul}
\liminf_{n\to\infty}\left\{\frac{\sum_{j=\lfloor n/2\rfloor}^{n-1}\Phi(\D_j)^2}{\log n}\right\}=\infty.
\end{align}
Then, there exists some positive $c_0=c_0(d, \gamma)$ such that $\P$-a.s. for all $n\ge (2^d-1)\ovl n(\omega)$,
\begin{align*}
\mathbf{P}^\omega(X_n=\mathbf 0)\ge \frac{c_0}{|\D^\omega_n|}.
\end{align*}
\end{ppn}
The random variable $\ovl n(\omega)$ is determined, thus condition (i) satisfied with $r(n)=(n/\omega_d)^{1/d}$, by Borel-Cantelli applied to the events $\{\cS^c_n\}$ of \eqref{shape-event}. Further, since $|\D_n|=n$, it follows from Theorem \ref{thm:cheeger} that $\P$-a.s. for all $n$ large enough, we have 
\begin{align*}
\sum_{j=\lfloor n/2\rfloor}^{n-1}\Phi(\D_j)^2 \ge c' n^{1-2/d-\ep}, 
\end{align*}
for some $c'>0$ and any $\ep\in(0,1-2/d)$, thereby verifying \eqref{regul} when $d\ge 3$. 

\medskip
We use here the {\it a-priori} weaker definition of recurrence. 
\begin{defn}
We say that $X$ is recurrent if it has infinite expected occupation time at the origin (hence at every other point), i.e. $\mathbf{E}\left[\int_1^\infty\mathbf{1}\{X_t=\mathbf 0\}dt\right]=\infty$. Otherwise, it is called transient.
\end{defn}

\begin{cor}   \label{rec}
For almost every realization $\omega$ of the discrete-time \abbr{IDLA} cluster $\{\D_n\}_{n\in\N}$ on $\Z^d$, $d\ge 3$, the uniformly lazy \abbr{SRW} $\{X_n\}_{n\in\N}$ evolving independently on $\{\D^\omega_n\}$ is recurrent.
\end{cor}

\begin{remark}
Our result does not cover $d=2$, for which we conjecture recurrence. In fact, a more general conjecture currently open asserts that {\it any sequence} of growing subgraphs $\{\G_n\}_{n\in\N}$ of $\Z^2$ is recurrent for the independently evolving \abbr{SRW}  $\{X_n\}_{n\in\N}$ \eqref{srw-def} on it, see \cite[Conjectures 1.10, 1.8]{DHS} or \cite[Conjectures 1.1, 7.1]{ABGK}. 
\end{remark}

\section{Proof}    \label{sec:proof}
For large enough $n=\lfloor\omega_dr^d\rfloor$, the \abbr{IDLA} cluster $\D_n$ concentrates on $\B_r$ with high probability by Theorem \ref{shape-fluct}. We thus denote 
\begin{align*}
\F_n:=\D_n\backslash\B_{r-a\sqrt{\log r}}
\end{align*}
(the part of fluctuation, when such event occurs). Our main goal is to show the following.
\begin{ppn}  \label{fluct-good}
Fix $d\ge 3$. There exist some finite constant $C_\star=C_\star(a,d)$ and integer $n_1$ such that for all $n=\lfloor\omega_dr^d\rfloor\ge n_1$, 
\begin{align}   \label{eq:rare}
\P\left(\exists   A\subset \cV(\F_n):\, |A|\ge C_\star|\partial^{\D_n}A|(\log r)^{3/2}\right)\le 3n^{-2}.
\end{align}
\end{ppn}

That this will be sufficient for proving Theorem \ref{thm:cheeger} is due  to a deterministic fact.

\begin{lem}  \label{decomp}
Suppose a connected subgraph $G\subset \Z^d$ is such that $G\supset \B_R$ for some $R>1$. Suppose there exists some positive constant $c_1$ such that $|\partial^GA|\ge c_1|A|(\log R)^{-3/2}$ holds for any $A\subset \cV(G\backslash\B_R)$. Then, there exists another positive constant $c_2=c_2(c_1,d)$ such that $|\partial^GA|\ge c_2 |A|R^{-1}(\log R)^{-3/2}$, for any $A\subset \cV(G)$. 
\end{lem}

\begin{proof}[Proof of Lemma \ref{decomp}]
Fixing $A\subset \cV(G)$ a vertex subset, we denote $A_1=A\cap\cV(\B_R)$ and $A_2= A\cap\cV(\B_R^c)$ such that $A=A_1\uplus A_2$.  We further denote $Q_2:=\{x\in \partial^G A_2:\, \exists y\in A_1  \text{ such that }x\sim y\}$ and $Q_1:=\{x\in A_1:\, \exists y\in A_2 \text{ such that } x\sim y\}$.  Trivially $|Q_2|\le (2d-1)|Q_1|$. 

Recall that by the optimal isoperimetry satisfied by lattice balls, $\Phi(\B_R)\ge c_dR^{-1}$ for some $c_d>0$ and any $R>1$, hence $|\partial^{\B_R}A_1|\ge c_dR^{-1}|A_1|$. Further, we can decompose $\partial^GA$ as follows (where $\uplus$ denotes disjoint union)
\begin{align*}
\partial^GA\supset\partial^{\B_R}A_1\uplus(\partial^GA_2 \backslash Q_2). 
\end{align*}
We separately discuss two cases. (a). If $|A_1|\ge |A_2|$, then clearly we have that 
\begin{align*}
|\partial^GA|\ge |\partial^{\B_R}A_1|\ge c_dR^{-1}|A_1|\ge \frac{1}{2}c_dR^{-1}|A|. 
\end{align*}
(b). Now consider $|A_1|<|A_2|$. Since 
\begin{align*}
|\partial^{\B_R}A_1|\ge c_dR^{-1}|A_1|\ge c_dR^{-1}|Q_1|\ge c_d(2d-1)^{-1}R^{-1}|Q_2|,
\end{align*}
we have that
\begin{align*}
|\partial^GA|&\ge |\partial^{\B_R}A_1|+|\partial^GA_2 \backslash Q_2|\ge c_d(2d-1)^{-1}R^{-1}|Q_2|+|\partial^GA_2 \backslash Q_2|  \\
&\ge c_d(2d-1)^{-1}R^{-1}|\partial^GA_2|\ge c_d(2d-1)^{-1}R^{-1}c_1|A_2|(\log R)^{-3/2}\\
&\ge \frac{1}{2}c_1c_d(2d-1)^{-1}|A|R^{-1}(\log R)^{-3/2}.
\end{align*}
Combining both cases proves the claim.
\end{proof}

Our proof relies on partitioning a discrete sphere into subsets of comparable size and diameter. There may be more than one way to achieve this. Here we utilize a notion of {\it almost regular partition} of the Euclidean unit sphere $S^{d-1}:=\{\mathsf{x}\in\R^d:\, \|\mathsf{x}\|_2=1\}$.
\begin{lem}[{\cite[Theorem 6.4.2]{DX}}]    \label{fair-par}
For every integer $M$, there exists a partition $\{\mathsf{T}_1, \mathsf{T}_2, ..., \mathsf{T}_M\}$ consisting of closed subsets of $S^{d-1}$ that satisfies the following:\\
(i). $S^{d-1}=\cup_{j=1}^M\mathsf{T}_j$ and $\mathsf{T}^\circ_i\cap \mathsf{T}^\circ_j=\emptyset$ whenever $i\neq j$, where $\mathsf{T}^\circ_j$ denotes the interior of $\mathsf{T}_j$. \\
(ii). $\sigma_d(\mathsf{T}_j)=M^{-1}\sigma_d(S^{d-1})$, for any $j=1,...,M$, where $\sigma_d(\cdot)$ denotes surface area measure.  \\
(iii). For some positive $c_p=c_p(d)$, the partition norm 
\begin{align*}
\max_{j=1}^M\max_{\mathsf{x,y}\in \mathsf{T}_j}\|\mathsf{x-y}\|_2 \le c_pM^{-\frac{1}{d-1}}.
\end{align*}
\end{lem}

Such area regular partition of $S^{d-1}$ naturally induces a partition of any discrete sphere.
\begin{defn}   \label{dis-par}
For a discrete sphere $\partial\B_R\subset\Z^d$, $R>1$, we call $\{T_1, T_2, ..., T_M\}$ its area regular partition, if $x\in T_j$ whenever $x/\|x\|_2\in \mathsf{T}_j$, for every $x\in\partial\B_R$ and $j=1,...,M$, where $\{\mathsf{T}_1, ..., \mathsf{T}_M\}$ is an area regular partition of $S^{d-1}$ in the sense of Lemma \ref{fair-par}. 
\end{defn}

We define the notion of a cone in $\R^d$ determined by the origin and  a closed subset of $S^{d-1}$.
\begin{defn}   \label{def:cone}
We say that $\mathsf{K}=\mathsf{K}(\mathsf{T})\subset\R^d$ is an infinite cone determined by a closed subset $\mathsf{T}\subset S^{d-1}$ and the origin, if 
\begin{align*}
\mathsf{K}(\mathsf{T}):=\{\mathsf{x}\in \R^d:\, \mathsf{x/\|x\|_2} \in\mathsf{T}\}. 
\end{align*}
\end{defn}

We then define the notion of a cell which is the intersection of a cone and an annulus in $\R^d$.
\begin{defn}     \label{def:cell}
We say that $\mathsf{C}=\mathsf{C}(\mathsf{A}, \mathsf{T})\subset\R^d$ is a cell determined by an annulus $\mathsf{A}=\mathsf{A}(R, R'):=\ovl{\mathbf{B}}_R\backslash\mathbf{B}_{R'}$, where $R>R'$, and a closed subset $\mathsf{T}\subset S^{d-1}$, if $\mathsf{C}=\mathsf{A}\cap\mathsf{K}$ where $\mathsf{K}=\mathsf{K}(\mathsf{T})$ is the infinite cone determined by $\mathsf{T}$ and the origin in the sense of Definition \ref{def:cone}.
\end{defn}

We also need a lemma for a specific hitting probability of (ordinary) \abbr{srw} on $\Z^d$. We defer its proof to the end of the section.
\begin{lem}   \label{lawler}
Fix $d\ge 3$. Let $\mathsf{C}\subset\R^d$ be a cell determined by the annulus $ \mathsf{A}(R+\delta\sqrt{\log R}, R)$  and a closed subset $\mathsf{T}\subset S^{n-1}$ with $\max_{\mathsf{x,y}\in \mathsf{T}}\|\mathsf{x-y} \|_2\le c_p/R$, as in Definition \ref{def:cell},
where $R, \delta, c_p>1$. We further denote $\C=\mathsf{C}\, \cap\, \Z^d$ the discrete cell, and $Y$ a \abbr{SRW} on $\Z^d$. Then, there exist  finite constants $R_1=R_1(d)$ and $C'=C'(d,\delta, c_p)$ such that for any $R>R_1$,
\begin{align}    \label{hit-cell}
\P_\mathbf 0(Y \text{ hits } \C \text{ before }\B^c_{R+\delta\sqrt{\log R}})\le C'\frac{\log R}{R^{d-1}},
\end{align}
where $\P_v$ indicates the law of $Y$ when starting from $v\in\Z^d$.
\end{lem}

We now prove the main proposition.

\begin{proof}[Proof of Proposition \ref{fluct-good}]
For brevity, we write 
\begin{align*}
n':=\lfloor\omega_d(r-2a\sqrt{\log r})^d\rfloor. 
\end{align*}
First note that, under the event $\cS_{n'}$ (see \eqref{shape-event}) we have $\F_n\cap\D_{n'}=\emptyset$ and $\D_{n'}\supset \B_{r-3a\sqrt{\log r}}$. Further, since $\cS_{n'}$ depends only on the first $n'$ injected particles, conditioning on $\cS_{n'}$ does not alter the independence of subsequent particles injections.  

By Lemma \ref{fair-par} and Definition \ref{dis-par}, there exists an area regular partition $\{\mathsf{T}_1, ..., \mathsf{T}_M\}$ of $S^{d-1}$  with $M=M(n):=\lfloor c'r^{d-1}\rfloor$ for any $c'>0$ chosen fixed, and its induced partition $\{T_1, ..., T_M\}$ of $\partial \B_{r-3a\sqrt{\log r}}$ such that $\|x-y\|_2\le c_p$ for some finite constant $c_p=c_p(d, c')>1$ and any $x, y\in T_j$, $j=1, ..., M$. Each partition block $\mathsf{T}_j$, $j=1,...,M$ and the annulus $\mathsf{A}(r+a\sqrt{\log r}, r-3a\sqrt{\log r})$ determines a cell $\mathsf{C}_j=\mathsf{C}_j(\mathsf{A}, \mathsf{T}_j)$ in the sense of Definition \ref{def:cell}. We then denote the discrete cells $\C_j:=\mathsf{C}_j\cap\Z^d$, $j=1,...,M$.

Let us focus on each  fixed cell $\C_j, \,  j\in\{1,...,M\}$. We compute the probability that the $m$-th particle $Y_m$, $m\in[n', n]$, upon being injected from the origin, hits $\C_j$ before $\D_m^c\backslash\C_{j}$, conditional on $\cS_{n'}$. To this end, we define the events
\begin{align}   \label{hit-event}
E^{(j)}_m:=\{Y_m \text{ hits }\C_j \text{ before }\D^c_m\backslash\C_{j}\}, \quad  n'<m\le n, \, j\in\{1,..,M\}.
\end{align}
For each $m\in[n', n]$, on the same probability space $(\Omega, \cF, \P)$ we define $Y_m^{\text{EXT}}$ to be a \abbr{SRW} evolving on $\Z^d$ obtained by continuing to run the \abbr{IDLA} particle $Y_m$ even after it has exited $\D_m$ (and hence settled).
Since $\D_m\subset\D_n$ and $\D_m^c\backslash\C_j\supset\B^c_{r+a\sqrt{\log r}}$, for $n'<m\le n$, we have the following inclusion of events
\begin{align}    \label{event-incl}
E^{(j)}_m\cap \{\D_n\subset \B_{r+a\sqrt{\log r}}\} \subset \wt E^{(j)}_m\cap \{\D_m\subset \B_{r+a\sqrt{\log r}}\}\subset \wt E^{(j)}_m,
\end{align}
where
\begin{align*}
\wt E^{(j)}_m:=\{Y^{\text{EXT}}_m \text{ hits }\C_j \text{ before }\B^c_{r+a\sqrt{\log r}}\}, \quad  n'<m\le n, \, j\in\{1,...,M\}.
\end{align*}
By Lemma \ref{lawler}, there exists some finite constant $C'=C'(d, a, c_p)$ such that for any $n'<m\le n$ and $j\in\{1,...,M\}$, 
\begin{align}    \label{bern-dom}
\P(\wt E_m^{(j)}\cap \{\D_n\subset \B_{r+a\sqrt{\log r}}\} \, | \, \cS_{n'})\le \P(\wt E^{(j)}_m \, | \, \cS_{n'})\le C'\frac{\log r}{r^{d-1}}.
\end{align}
Since the particles trajectories $\{Y_m: n'<m\le n\}$ are independent prior to settling, in view of \eqref{event-incl}, for each fixed $j\in\{1,...,M\}$, the events $\{E^{(j)}_m\cap \{\D_n\subset \B_{r+a\sqrt{\log r}}\}:\, n'<m\le n\}$ (when conditioned on $\cS_{n'}$) are stochastically dominated from above by a sequence of independent Bernoulli's with success probability at most the \abbr{RHS} of \eqref{bern-dom}. Since also $n-n'=c_*(d,a)r^{d-1}\sqrt{\log r}$ for some explicit constant $c_*(d,a)\in(0, \infty)$, we have by Chernoff bound
\begin{align}     \label{total-pass}
\P\Big[\Big\{\sum_{m=n'+1}^n\mathbf{1}\{E^{(j)}_m\}\ge 2C'c_\star(\log r)^{3/2}\Big\}\cap\{\D_n\subset \B_{r+a\sqrt{\log r}}\}\, |\, \cS_{n'}\Big]\le e^{-C'c_\star(\log r)^{3/2}}. 
\end{align}

Since $A\subset\cV(\F_n)$ and $\F_n\cap\D_{n'}=\emptyset$, the particle $Y_k$ that eventually settles on some $x\in A$ must have index $n'<k=k(x)\le n$. Further, since $\partial^{\D_n}A$ is a vertex-cutset that separates $A$ from the origin, such $Y_k$ must pass through $\partial^{\D_n}A$ at least once before settling on $x\in A$. Let us denote by $v=v(x)$ the first vertex in $\partial^{\D_n}A$ that $Y_k$ hits, before settling at $x$. Since $v\in\cV(\F_n)$, it belongs to some cell $\C_{j(v)}$, where $j(v)\in\{1,...,M\}$. Consequently, such $Y_k$ in order to settle on $x\in A$ must hit $\C_{j(v)}$ before $\D_k^c\backslash\C_{j(v)}$ (since otherwise $Y_k$ already settles on the first vertex it encounters in $\D_k^c\backslash\C_{j(v)}$ and will not hit $v$ hence will not settle on $x$).

With $A\subset\cV(\F_n)$, $\partial^{\D_n}A$ cannot intersect more than $2|\partial^{\D_n}A|\wedge M$ number of cells, as the cells are vertex-disjoint except possibly at their boundaries.  Further, every $x\in A$ is settled by some particle $Y_{k(x)}$.
Thus, the event $\{\exists A\subset\cV(\F_n):\, |A|\ge 4C'c_\star|\partial^{\D_n}A|(\log r)^{3/2}\}$ implies, by pigeonhole principle, that there exists at least one cell, say $\C_{j*}$, for which the event $\{\sum_{m=n'+1}^n\mathbf{1}\{E^{(j^*)}_m\}\ge 2C'c_\star(\log r)^{3/2}\}$ occurs.  Due to in total $M=\lfloor c'r^{d-1}\rfloor$ number of cells, we obtain by \eqref{total-pass} and union bound that
\begin{align*}
&\P\big(\{\exists A\subset \cV(\F_n):\, |A|\ge 4C'c_\star|\partial^{\D_n}A|(\log r)^{3/2}\} \cap\{\D_n\subset \B_{r+a\sqrt{\log r}}\} \, |\, \cS_{n'}\big) \\
&=\P\Big(\Big\{ \exists j^*\in\{1,..,M\}:\, \sum_{m=n'+1}^n\mathbf{1}\{E^{(j^*)}_m\}\ge 2C'c_\star(\log r)^{3/2} \Big\}\cap\{\D_n\subset \B_{r+a\sqrt{\log r}}\} \, |\, \cS_{n'}\Big)  \\
&\le \sum_{j=1}^M\P\Big[\Big\{\sum_{m=n'+1}^n\mathbf{1}\{E^{(j)}_m\}\ge 2C'c_\star(\log r)^{3/2}\Big\}\cap\{\D_n\subset \B_{r+a\sqrt{\log r}}\}\, |\, \cS_{n'}\Big]    \\
& \le c'r^{d-1}e^{-C'c_*(\log r)^{3/2}}.
\end{align*}
Noticing that $\{\D_n\subset \B_{r+a\sqrt{\log r}}\}^c\subset \cS_n^c$, we can further bound
\begin{align*}
&\P(\exists A\subset \cV(\F_n):\,  |A|\ge 4C'c_\star|\partial^{\D_n}A|(\log r)^{3/2}) \\
  &\le \P(\cS_{n'}^c\cup\cS_{n}^c)\, +   
\P\big(\{\exists A\subset \F_n:\, |A|\ge 4C'c_\star|\partial^{\D_n}A|(\log r)^{3/2}\}\cap\{\D_n\subset \B_{r+a\sqrt{\log r}}\} \, |\, \cS_{n'}\big).
\end{align*}
Upon taking $n$ large enough such that $n'\ge n_0$, by Theorem \ref{shape-fluct} we have that
\begin{align*}
\P(\exists A\subset \cV(\F_n):\, & |A|\ge 4C'c_\star|\partial^{\D_n}A|(\log r)^{3/2}) \le 2n^{-2}+c'r^{d-1}e^{-C'c_*(\log r)^{3/2}}.
\end{align*}
which can be made smaller than $3n^{-2}$ upon taking $n\ge n_1$ for some large enough $n_1$. 
\end{proof}

\begin{proof}[Proof of Theorem \ref{thm:cheeger}]
By Proposition \ref{fluct-good} and Theorem \ref{shape-fluct}, there exist positive constant $c_1=c_1(a,d)=1/C^\star$ and finite integer $n_1$, such that for all $n=\lfloor\omega_dr^d\rfloor\ge n_1$,
\begin{align*}
\P\left(\big\{\forall A\subset \cV(\F_n): \, |\partial^{\D_n}A|\ge c_1|A|(\log r)^{-3/2}\big\}\cap\cS_n\right)\ge 1-4n^{-2}.
\end{align*}
Combining with Lemma \ref{decomp}, there exists another positive constant $c_2=c_2(c_1, d)$ such that
\begin{align*}
\P\left(\forall A\subset \cV(\D_n): \, |\partial^{\D_n}A|\ge c_2|A|r^{-1}(\log r)^{-3/2}\right)\ge 1-4n^{-2}.
\end{align*}
By Borel-Cantelli, the \abbr{LHS} event occurs a.s. for all large enough $n$, namely \eqref{lower-ch} holds.
\end{proof}

We are left to prove Lemma \ref{lawler}. First, we clarify some additional notations. Since we work here with \abbr{SRW} $Y$ on $\Z^d$, without incurring ambiguity we identity any subgraph with its vertex set. Further, we denote by $\tau(A)$ and $H_A(\cdot, \cdot)$, respectively, the first exit time (discounting time $0$) and exit distribution of $A\subset \Z^d$ by a \abbr{SRW}:
\begin{align*}
\tau(A):=\inf\{k\ge 1:\, Y_k\in A^c\}, \quad \quad H_A (x,y):=\P(Y_{\tau(A)}=y|Y_0=x).
\end{align*}
Recall that Green's function on $\Z^d$, $d\ge 3$, is harmonic except at the origin, and satisfies the asymptotics as $\|x\|_2\to\infty$
\begin{align}   \label{green}
\mathsf G(x)=\E_x\Big[\sum_{k=0}^\infty \mathbf{1}\{Y_k=\mathbf 0\}\Big]=\mathsf C_d\|x\|_2^{2-d}+O(\|x\|_2^{1-d}), \quad  \text{where} \quad  \mathsf C_d:=\frac{2}{(d-2)\omega_d}.
\end{align}
One also defines Green's function restricted to $A\subset \Z^d$ as the symmetric function 
\begin{align*}
\mathsf G_A(x,y)=\E_x\Big[\sum_{k=0}^{\tau(A)-1}\mathbf{1}\{Y_k=y\}\Big].
\end{align*}

\begin{proof}[Proof of Lemma \ref{lawler}]
The proof is similar to \cite[Lemma 6.3.7]{LL}. By scaling relations, the property of $\mathsf{T}\subset S^{d-1}$ transfers to the fact that any two vertices on the lower base of the discrete cell $\C$ (on $\partial\B_R$) satisfies $\|x-y\|_2\le c_p$.  In particular, the lower base of the cell is contained in a spherical cap of radius $c_p$ on $\partial\B_R$. Denoting $\K:=\B_{R+\delta\sqrt{\log R}}\backslash \C$, we have that $\B_R\subset \K\subset \B_{R+\delta\sqrt{\log R}}$.

It is a consequence of last exit decomposition, see \cite[Lemma 6.3.6]{LL}, that for any $x\in \partial (\K^c)$, i.e. the {\it outer} vertex boundary of $\K$,
\begin{align}   \label{hit-dist}
H_{\K}(\mathbf 0,x)=\sum_{z\in \B_{R/2}}\mathsf G_{\K}(z,\mathbf 0)\P_x(Y_{\tau(\K\backslash\B_{R/2})}=z)=\sum_{z\in \partial\B_{R/2}}\mathsf G_{\K}(z, \mathbf 0)\P_x(Y_{\tau(\K\backslash\B_{R/2})}=z).
\end{align}
Uniformly for all $z\in\partial \B_{R/2}$, by \cite[Proposition 6.3.5]{LL} we have that
\begin{align*}
\mathsf G_{\K}(z,\mathbf 0)\le \mathsf G_{\B_{R+\delta\sqrt{\log R}}}(z, \mathbf 0) &\le \mathsf C_d\big[(R/2)^{2-d}-(R+\delta\sqrt{\log R})^{2-d}\big]+O(R^{1-d})\\
&\le \mathsf C_d 2^{d-1} R^{2-d}.
\end{align*}
The first inequality is due to both vertices $\mathbf 0, z\in\K\subset\B_{R+\delta\sqrt{\log R}}$, and 
the last inequality holds provided we take $R> R_1$ for some large enough $R_1=R_1(d)$. Further, for any $x\in \partial(\K^c)$,
\begin{align*}
\sum_{z\in\partial\B_{R/2}}\P_x(Y_{\tau_{\K\backslash\B_{R/2}}}=z) &=\P_x(Y \text{ hits }\B_{R/2} \text{ before }\K^c)\\
&\le \max_{y\in\partial\B_{R}}\left\{\P_y(Y \text{ hits }\B_{R/2} \text{ before }\K^c)\right\},
\end{align*}
since  starting from $\K^c$, in order to hit $\B_{R/2}$, $Y$ must first pass through $\partial\B_{R}$. Since $\{\mathsf G(Y_k)\}_{k\in\N}$ is a martingale provided $Y$ is bounded away from $\mathbf 0$, by optional stopping theorem and \eqref{green}, the latter probability is bounded by
\begin{align*}
\max_{y\in\partial\B_{R}}\left\{\P_y(Y \text{ hits }\B_{R/2} \text{ before }\B^c_{R+\delta\sqrt{\log R}})\right\}
&\le 2\frac{R^{2-d}-(R+\delta\sqrt{\log R})^{2-d}}{(R/2)^{2-d}-(R+\delta\log R)^{2-d}}\\
&\le 2^{d-1}(d-2)\delta\sqrt{\log R}/R,
\end{align*}
when $R>R_1$, upon increasing $R_1$ if necessary. Plugging these estimates into \eqref{hit-dist}, uniformly for all $x\in\partial (\K^c)$,  we have that $H_\K(\mathbf 0, x) \le \mathsf C_d 4^{d-1}(d-2)\delta R^{1-d}\sqrt{\log R}$.

Now, since $\partial\C\subset\partial(\K^c)$, the \abbr{LHS} of \eqref{hit-cell} is the same as $\sum_{x\in \partial\C}H_{\K}(\mathbf 0,x)$. The cardinality of $\partial\C$ is at most $c(d)c_p^{d-1}\delta\sqrt{\log R}$ for some dimensional constant $c(d)$. We thus get the desired bound of the form $C'R^{1-d}\log R$ for some $C'=C'(d,\delta, c_p)$ finite. 
\end{proof}

\section{Extensions to continous-time \abbr{IDLA}}     \label{sec:cont}

We consider simple random walk $\{X_{\lfloor t\rfloor}\}_{t\ge 0}$ \eqref{srw-def} evolving independently on a continuous-time analogue of \abbr{IDLA} $\{\D_t\}_{t\ge 0}$ on $\Z^d$, $d\ge 3$, formed by particles injected from the origin at an inhomogeneous Poisson rate $\lambda(t)>0$. As a result, one can achieve quite general growth rate for the cluster by tuning the particles injection rate. Such continuous-time extension of the classical model appears e.g. in \cite{LBG, GQ}.

More specifically, let $\{N(t)\}_{t\ge 0}$ be an inhomogeneous Poisson process with mean $m(t):=\int_0^t\lambda(s)ds$, satisfying suitable conditions as specified in \eqref{cond:rate}. Similarly to the discrete-time case, starting at $\D_0=\{\mathbf 0\}$, at each Poisson arrival time, a particle is injected from the origin, performing \abbr{SRW} until settling on the first unoccupied vertex, and then that vertex is added to the cluster. Clearly, by time $t$, a total of $N(t)$ particles have been injected. We still omit the travel time of the particles and consider them settled instantaneously, therefore $|\D_t|=N(t)$. Given $\{\D_t\}_{t\ge 0}$ we perform (lazy) \abbr{SRW} $\{X_{\lfloor t\rfloor}\}_{t\ge 0}$ independently on the cluster, jumping at integer times, as specified in \eqref{srw-def} with $n=\lfloor t\rfloor$. 

(This omission of particles travel time is inconsequential when $d\ge 3$, as explained already in \cite[Section 6]{LBG} and \cite{GQ}. That is,  our conclusions are unaltered if we allow particles to move simultaneously with each other and with the autonomous random walk $X$. When $d=2$, there is indeed a difference between the two formalisms.)

Since both Theorem \ref{shape-fluct} and Theorem \ref{thm:cheeger} are properties of fixed large cluster,  irrespective of whether the \abbr{IDLA} emission is time-continuous or discrete, we have the following result.
\begin{cor}    \label{phase-transition}
Fix $d\ge 3$. Assume that there exist some constants $0\le \alpha<d/2$ and $\beta>0$, such that the following hold for all $t$ large enough
\begin{align}   \label{cond:rate}
\beta\log t\le m(t)\le t^{\alpha}.
\end{align}
Then, for almost every realization $\omega$ of the continuous-time \abbr{IDLA} cluster $\{\D_t\}_{t\ge 0}$ on $\Z^d$, the uniformly lazy \abbr{SRW} $\{X_{\lfloor t\rfloor}\}_{t\ge 0}$ evolving independently on $\{\D^\omega_t\}_{t\ge 0}$ is recurrent as soon as $\int_1^\infty |\D^\omega_t|^{-1}dt=\infty$. Further, $X$ is recurrent for a.e. realization $\omega$ as soon as $\int_1^\infty m(t)^{-1}dt=\infty$.
\end{cor}

\begin{remark}
Transience under $\int_1^\infty m(t)^{-1}dt<\infty$ is proved in \cite[Corollary 1.6, Remark 1.7]{DHS} for the {\it a-priori } weaker definition of transience as ``almost surely finitely many returns'' (instead of finite expected occupation time). The method used is different from this article.
\end{remark}

\begin{proof}[Proof of Corollary \ref{phase-transition}]
Since $N(t)$ is Poisson distributed with mean $m(t)=\int_0^t\lambda(s)ds$, for any $\eta>0$ and all $t> 1$ we have that
\begin{align}   \label{poisson-concentration}
\P\left(N(t)>(1+\eta) m(t)\right) \le \exp\Big\{-\frac{\eta^2m(t)}{2(1+\eta)}\Big\}\le \exp\Big\{-\frac{\beta\eta^2\log t}{2(1+\eta)}\Big\}.
\end{align}
Given $\beta>0$, there exists $\eta=\eta_\beta<\infty$ so that $\beta\eta^2/(2(1+\eta))>1$, which ensures that \abbr{RHS} of \eqref{poisson-concentration} is summable and $\P$-a.s. the \abbr{LHS} events of \eqref{poisson-concentration} at integer times $\lfloor t\rfloor$ occurs only finitely often, by Borel-Cantelli. Since $m(t)\to\infty$ as $t\to\infty$, the condition (i) of Proposition \ref{hk-lbd} is satisfied when $t$ is large enough, hence to apply that result it remains to verify again condition (ii) therein. Since a.e. $\omega$, for all $\lfloor t\rfloor$ large enough $|\D^\omega_{\lfloor t\rfloor}|=N^\omega(\lfloor t\rfloor)\le (1+\eta_\beta)m(\lfloor t\rfloor)$, by Theorem \ref{thm:cheeger} and \eqref{cond:rate} we have that
\begin{align*}
\sum_{k=\lfloor t/2\rfloor}^{\lfloor t \rfloor-1}\Phi^\omega(\D_k)^2 &\ge \sum_{k=\lfloor t/2\rfloor}^{\lfloor t \rfloor-1}N^\omega(k)^{-2/d}(\log N^\omega(k))^{-3/2}\ge c(d, \beta)\sum_{k=\lfloor t/2\rfloor}^{\lfloor t \rfloor-1}m(k)^{-2/d}(\log m(k))^{-3/2}\\
&\ge c(d,\beta, \alpha)\sum_{k=\lfloor t/2\rfloor}^{\lfloor t \rfloor-1}k^{-2\alpha/d}(\log k)^{-3/2},
\end{align*}
where the constant changes from line to line.
With $0\le \alpha<d/2$, there exists $\ep=\ep(d, \alpha)>0$ such that $1-2\alpha/d-\ep>0$ and thus
\begin{align*}
\sum_{k=\lfloor t/2\rfloor}^{\lfloor t \rfloor-1}\Phi^\omega(\D_k)^2 \ge c(d, \beta, \alpha)t^{1-2\alpha/d-\ep}\gg\log t.
\end{align*}
Therefore, Proposition \ref{hk-lbd} applies to give for some $c_0=c_0(d, \gamma)>0$, a.e. $\omega$,  and all $t$ sufficiently large
\begin{align*}
\mathbf{P}^\omega(X_t=\mathbf 0)\ge \frac{c_0}{|\D_{\lfloor t\rfloor}^\omega|}\ge \frac{c_0}{(1+\eta_\beta)m(\lfloor t\rfloor)}\ge \frac{c_0}{(1+\eta_\beta)m(t)}.
\end{align*}
This immediately implies the results.
\end{proof}

\medskip
\noindent
{\bf{Acknowledgments.}} I thank Eviatar Procaccia for pivotal comments that led to this work, and Amir Dembo for fruitful discussions.

\end{document}